\documentclass[a4paper,11pt]{amsart}

\usepackage[utf8]{inputenc}
\usepackage[english]{babel}
\usepackage{mathrsfs}
\usepackage{amsfonts, amssymb, amsmath, amsthm}
\usepackage[colorlinks=true,pdfauthor={Nicol\string\341s Sirolli},pdftitle={Preimages for the Shimura map on Hilbert modular forms},pdfkeywords={Shimura map, Hilbert modular forms, half-integral weight Hilbert modular forms}]{hyperref}
\usepackage{url}
\usepackage[all]{xy}
\usepackage{a4wide}

\newcommand{\Q}{\mathbb{Q}}
\newcommand{\R}{\mathbb{R}}
\newcommand{\A}{\mathbb{A}}
\newcommand{\C}{\mathbb{C}}
\newcommand{\Z}{\mathbb{Z}}
\newcommand{\HH}{\mathcal{H}}
\newcommand{\OO}{\mathcal{O}}
\newcommand{\pp}{\mathfrak{p}}

\newcommand{\qq}{\mathfrak{q}}
\newcommand{\D}{\mathfrak{D}}
\newcommand{\cl}{\mathit{Cl}}
\newcommand{\TT}{\mathbb{T}}

\newcommand{\SL}{\operatorname{SL}}
\newcommand{\GL}{\operatorname{GL}}

\newcommand{\Tr}{\operatorname{Tr}}

\newcommand{\Shim}{\operatorname{Shim}}

\def\ll#1{{\left\langle{#1}\right\rangle}}
\def\id#1{{\mathfrak{#1}}}

\numberwithin{equation}{section}

\theoremstyle{plain}
\newtheorem{proposition}[equation]{Proposition}
\newtheorem{theorem}[equation]{Theorem}
\newtheorem*{thm*}{Theorem}
\newtheorem*{definition}{Definition}

\newtheorem{lemma}[equation]{Lemma}

\theoremstyle{remark}

\newtheorem{remark}[equation]{Remark}
\newtheorem{conjecture}{Conjecture}

\begin{document}

\title{Preimages for the Shimura map on Hilbert modular forms}
\author{Nicol\'as Sirolli}
\email{nsirolli@dm.uba.ar}
\address{Departamento de Matem\'atica, Facultad de Ciencias Exactas y Naturales, Universidad de Buenos Aires. Ciudad Universitaria (C1428EGA), Buenos Aires, Argentina}

\thanks{The author was partially supported by a CONICET PhD Fellowship.}

\begin{abstract}

In this article we give a method to construct preimages for the Shimura correspondence on Hilbert modular forms of odd and square-free level. The method relies in the ideas presented for the rational case by Pacetti and Tornar{\'{\i}}a, and is such that the Fourier coefficients of the preimages constructed can be computed explicitly.

\end{abstract}

\maketitle

\section*{Introduction}\label{sec:intro}

The Shimura map is a Hecke linear map between half-integral weight modular forms and integral weight ones, introduced in \cite{shim-halfint} in the classical setting and generalized in \cite{shim-hhalfint} to Hilbert modular forms, as well as to the automorphic setting by the work of Waldspurger, Flicker and others. Computing preimages for the Shimura map became an interesting subject, after the formulas given by Waldspurger \textit{et al.} relating the central values of twists of the $L$-series associated to an integral weight modular form $f$ with the Fourier coefficients of a half-integral weight form $g$ mapping to $f$ by the Shimura map. Such formulas have been generalized to the Hilbert setting in \cite{shim-coef} and \cite{mao}.

The problem of computing preimages for the Shimura map in the classical setting has been considered, for example, in \cite{shintani} and \cite{gross}. Our method for computing preimages in the Hilbert setting relies in the ideas present in \cite{tornaria}, which in turn generalize the method of Gross. The preimages are obtained by considering certain ternary theta series associated to ideals in quaternion algebras. The problem of computing these ideals is thus crucial for our method, and has been studied in \cite{dem-voi} and \cite{ariel-yo}.

The correspondence between ideals in quaternion algebras and half-integral weight modular forms has its automorphic counterpart, and was studied by  \cite{waldspurger} over any number field, and in particular in the Hilbert setting. The advantage of our method is that, being more explicit, it allows to compute effectively the Fourier coefficients of the preimages.

In \cite{xue}, the author also follows the method of Gross for computing half-integral weight Hilbert modular forms to prove a Waldspurger's type formula, but with several restrictions such as working with prime level and odd class number of the base field, and with no focus on Hecke operators nor the Shimura correspondence.

\medskip

We start this article by recalling basic definitions regarding Hilbert modular forms and setting some notation. Some good references for the theory of Hilbert modular forms are Garrett's book \cite{garrett} and Shimura's article \cite{shim-special}.

In the second section, given a totally definite quaternion algebra $B$ and an Eichler order $R$ in it, we define Hecke operators acting on the space $M(R)$ generated by left ideal classes representatives for $R$, showing that they satisfy properties analogous to those of the Hecke operators on Hilbert modular forms.

In the third section we introduce half-integral weight Hilbert modular forms, following \cite{shim-hhalfint}. We state the main properties of the Hecke operators acting on them, and we recall Shimura's correspondence.

In the fourth section we show how certain ternary theta series associated to the left ideal classes of a given order $R$ can be used to produce Hilbert modular forms of parallel weight $3/2$, thus giving a Hecke linear map from the space $M(R)$ to the space of Hilbert modular forms of parallel weight $3/2$.

In the fifth section we use the results from the previous sections to construct preimages of the Shimura map, at least in the case where the level of the modular form is odd and square-free. This is stated in Theorem~\ref{thm:teo_ppal}, which is our main result. We also state a Waldspurger's type formula relating the Fourier coefficients of the preimages and central values of twisted $L$-functions.

In the final section we consider the space of Hilbert modular cusp forms over $F=\Q[\sqrt{5}]$, with level $(6+\sqrt{5})$ and parallel weight $2$. This space is $1$-dimensional, and it is spanned by a newform that corresponds to an elliptic $E$ curve over $F$. We apply our method to this cusp form to construct a parallel weight $3/2$ modular form in Shimura correspondence with it, and compare its zero coefficients with the ranks of imaginary quadratic twists of $E$.

\medskip

We remark that though for simplicity we consider the Shimura correspondence in parallel weights $3/2$ and $2$, our techniques can be used for general weights, adding spherical polynomials to the ternary theta series.

\section{Hilbert modular forms}\label{sec:1}

Let $F$ be a totally real number field of degree $d$ over $\Q$, with different ideal $\id{d}$. We let $\mathbf{a}$ denote the set of all embeddings $\tau:F\hookrightarrow\R$, and for $\xi\in F$ and $\tau\in\mathbf{a}$, we denote $\tau(\xi)=\xi_\tau$. We let $F^+$ denote the set of $\xi\in F$ such that $\xi_\tau>0$ for all $\tau\in\mathbf{a}$, and we let
\[
 \GL_2^+(F)=\{\gamma \in \GL_2(F) : \det \gamma \in F^+\}.
\]

Let $\HH$ denote the Poincar\'e upper-half plane. The group $\GL_2^+(\R)^\mathbf{a}$ acts on $\HH^\mathbf{a}$ component-wise, and $\GL_2^+(F)$ also acts on $\HH^\mathbf{a}$ via the natural embedding $\GL_2^+(F) \hookrightarrow \GL_2^+(\R)^\mathbf{a}$. If $\gamma \in \GL_2^+(\R)^\mathbf{a}$, with $\gamma_\tau =\big(\begin{smallmatrix} a_\tau & b_\tau \\ c_\tau &d_\tau \end{smallmatrix}\big)$, we let $j(\gamma,z)$ denote the automorphy factor
\[
 j(\gamma,z)=\prod_{\tau\in\mathbf{a}} (c_\tau z_\tau+d_\tau).
\]
Again, this also makes sense for $\gamma \in \GL_2^+(F)$. Given a function $g:\HH^\mathbf{a}\to\C$ and $\gamma \in \GL_2^+(F)$, we denote by $g\vert\gamma$ the function given by 
\[
 (g\vert\gamma)(z)=N_{F/\Q}(\det\gamma) j(\gamma,z)^{-2} g(\gamma z).
\]

Let $\OO$ be the ring of integers of $F$. Given fractional ideals $\id{r},\id{n}$, let
\[
  \tilde\Gamma[\id{r},\id{n}] = \big\{\gamma=\big(\begin{smallmatrix} a & b \\ c & d \end{smallmatrix}\big) \in\GL_2^+(F):a,d\in\OO,b\in \id{r}^{-1},c\in\id{r}\id{n},\det\gamma\in\OO^\times\big\}.
\]
The space of Hilbert modular forms of weight $\mathbf{2}$ with respect to $\tilde\Gamma[\id{r},\id{n}]$, which we denote by $M_\mathbf{2}(\tilde\Gamma[\id{r},\id{n}])$, is the space of holomorphic functions $g:\HH^\mathbf{a}\to\C$ such that
\begin{itemize}
\item $g\vert\gamma=g\quad \forall\, \gamma\in\tilde\Gamma[\id{r},\id{n}]$.
\item If $d=1$, $g(z)$ is holomorphic at the cusps.
\end{itemize}

Let $e_F:F\times \HH^\mathbf{a} \to \C$ be the exponential function given by
\[
 e_F(\xi, z)=\exp{\big(2\pi i\sum_{\tau\in\mathbf{a}} \xi_\tau z_\tau\big)}.
\]
For a fractional ideal $\id{a}$, let $\id{a}^+=\id{a}\cap F^+$, and denote by $\id{a}^\vee$ its dual with respect to the trace form. Then every $g\in M_\mathbf{2}(\tilde\Gamma[\id{r},\id{n}])$ has a Fourier series expansion
\[
g(z)=\sum_{\xi\in ((\id{r}^{-1})^\vee)^+\cup\{0\}} c(\xi,g) e_F\big(\xi,z\big).
\]
We say that $g$ is \emph{cuspidal} if $c(0,g\vert\gamma)=0$ for all $\gamma\in\GL_2^+(F)$. The subspace of such $g$ is denoted by $S_\mathbf{2}(\tilde\Gamma[\id{r},\id{n}])$.

\medskip

Take $\id{b}_1,\dots,\id{b}_r\subseteq\OO$ representatives for the narrow class group $\cl^+(F)$. Let $\id{c}$ be an integral ideal. The spaces of \emph{Hilbert modular forms} and \emph{Hilbert modular cusp forms} of level $\id{c}$ are defined respectively by
\[
 M_\mathbf{2}(\id{c})=\bigoplus_{l=1}^r
M_\mathbf{2}(\tilde\Gamma[\id{b}_l,\id{c}]), \quad
S_\mathbf{2}(\id{c})=\bigoplus_{l=1}^r
S_\mathbf{2}(\tilde\Gamma[\id{b}_l,\id{c}]).
\]

Every $g\in M_\mathbf{2}(\id{c})$ has attached Fourier coefficients indexed by integral ideals. Given a non-zero integral ideal $\id{m}$, we let
\[
c(\id{m},g)=c(\xi,g_l),\quad \text{ with } \xi\in\id{b}_l^+ \text{ such that } \id{m}=\xi\id{b}_l^{-1},
\]
and this is well defined.

\section{Quaternionic modular forms}\label{sec:2}

We refer to \cite{vig} for the definitions and basic results concerning the arithmetic of quaternion algebras.

Let $B$ a totally definite quaternion algebra over $F$, i.e. such that $B_\tau=B\otimes_F F_\tau$ is ramified quaternion algebra over $F_\tau$ for every $\tau\in\mathbf{a}$. We fix an Eichler order $R$ in $B$ of discriminant $\D$, and we let $\mathfrak{I}(R)$ denote the set of invertible (i.e., locally principal) left $R$-ideals. Given an ideal $I$, we denote its right order by $R_r(I)$. Given a prime ideal $\pp$ of $F$, by $B_\pp,R_\pp,I_\pp$ we denote the different completions at $\pp$. We denote by $N$ the reduced norm in $B$ (and in $B_\pp$).

Two ideals $I,J\in \mathfrak{I}(R)$ are equivalent if there exists $x\in B^\times$ such that $I=Jx$. We denote by $[I]$ the equivalence class of $I$ under this relation, and we denote by $\cl(R)$ the set of equivalence classes.

The space of \emph{quaternionic modular forms} for $R$ is the vector space over $\C$ spanned by $\cl(R)$, and is denoted by $M(R)$. On $M(R)$ we consider the inner product defined by
\[
 \ll{[I],[J]}=\#\{x\in \OO^\times\backslash B^\times: I x = J\}=\begin{cases}
                   [R_r(I)^\times:\OO^\times], & [I]=[J],\\
                   0, & [I]\neq[J],
                  \end{cases}
\]
where $[R_r(I)^\times:\OO^\times]$ denotes the (finite) index of $\OO^\times$ in $R_r(I)^\times$.

Given a fractional ideal $\id{a}$, denote by $[\id{a}]$ its equivalence class in the class group $\cl(F)$. For each $[\id{a}]\in\cl(F)$, denote $\cl_{[\id{a}]}(R)=\{[I]\in\cl(R):[N(I)]=[\id{a}]\}$. We let
\[
e_{[\id{a}]}=\sum_{[I]\in\cl_{[\id{a}]}(R)}\frac{1}{\ll{[I],[I]}} [I] \quad \in M(R).
\]
We let $E(R)=\ll{e_{[\id{a}]}:[\id{a}]\in\cl(F)}_{\C}$. The orthogonal complement of $E(R)$ is denoted by $S(R)$ and is called the space of \emph{quaternionic cusp forms}.

\medskip

Let $\id{m}$ be a non-zero integral ideal. For $I\in \mathfrak{I}(R)$ denote
\[
 t_\id{m}(I)=\{J\in \mathfrak{I}(R): J\subseteq I,[I:J]=\id{m}^2\},
\]
where $[I:J]$ denotes the index of $J$ in $I$. We let $T_\id{m}$ be the $\id{m}$-th Hecke operator acting on $M(R)$, defined by
\[
 T_\id{m}([I])=\sum_{J\in t_\id{m}(I)} [J].
\]

\begin{lemma}\label{lemma:nro-orbitas}
 Let $\pi_\pp$ denote a local uniformizer at $\pp$. Let $x_\id{p}\in M_2(\OO_\id{p})$ with $\pi_\id{p}\mid\det(x_\id{p})$. Then,
\begin{align*}
  \#\SL_2(\OO_\id{p}) & \backslash\big\{y_\id{p}\in
M_2(\OO_\id{p}):\det(y_\id{p})=\pi_\id{p},x_\id{p}y_\id{p}^{-1}\in
M_2(\OO_\id{p})\big\} \\ & =
\begin{cases}
1, &  x_\id{p}\notin \pi_\pp M_2(\OO_\id{p}), \\
N(\id{p})+1, & x_\id{p}\in \pi_\pp M_2(\OO_\id{p}).
\end{cases}
\end{align*}

\end{lemma}

\begin{proof}
 Let $q=N(\id{p})$, and let $\alpha_1,\dots,\alpha_q\in\OO$ be representatives for the residual classes modulo $\id{p}$. Then
\[
 \left(\begin{matrix} \pi_\id{p} & 0 \\ 0 & 1\end{matrix}\right),
 \left(\begin{matrix} 1 & \alpha_1 \\ 0 & \pi_\id{p}\end{matrix}\right),\dots,
 \left(\begin{matrix} 1 & \alpha_q \\ 0 & \pi_\id{p}\end{matrix}\right)
\]
is a system of representatives for the action of $\SL_2(\OO_\id{p})$ on $\{y_\id{p}\in M_2(\OO_\id{p}):\det(y_\id{p})=\pi_\id{p}\}$ by left multiplication. The result follows then from the fact that, given $x_\id{p}=\left(\begin{smallmatrix} a & b \\ c & d \end{smallmatrix}\right)\in M_2(\OO_\id{p})$,
\begin{align*}
x_\id{p} \left(\begin{matrix} \pi_\id{p} & 0 \\ 0 & 1\end{matrix}\right)^{-1}
\in M_2(\OO_\id{p})& \Longleftrightarrow \pi_\id{p}\mid a, \pi_\id{p}\mid c, \\
x_\id{p} \left(\begin{matrix} 1 & \alpha \\ 0 & 
\pi_\id{p}\end{matrix}\right)^{-1} \in M_2(\OO_\id{p}) & \Longleftrightarrow  
\pi_\id{p}\mid b-\alpha a, \pi_\id{p}\mid d-\alpha c.
\end{align*}
\end{proof}

The Hecke operators on $M(R)$ satisfy the following equalities, which are also satisfied by the Hecke operators on Hilbert modular forms (see \cite[(2.12)]{shim-special}).

\begin{proposition}\label{prop:hecke-cuat}
Let $\id{m},\id{n}$ be integral ideals, and let $\pp$ be a prime ideal such that $\pp\nmid\D$. The Hecke operators on $M(R)$ satisfy:
\begin{enumerate}
 \item $T_\id{m} T_\id{n} = T_{\id{m}\id{n}}$, if $(\id{m}:\id{n})=1$.
 \item $T_{\pp^{k+2}}=T_{\pp^{k+1}} T_\pp - N(\pp) \pp T_{\pp^k}$, for every $k\geq 0$.
 \item $T_\id{m} T_\pp = T_{\id{m}\pp}+N(\pp) \pp T_{\id{m}/\pp}$, if $\id{p}\mid\id{m}$.
\end{enumerate}
\end{proposition}

\begin{proof} 

\begin{enumerate} 
 \item Let $I\in\mathfrak{I}(R)$. If $J\in t_\id{m}(L)$ with $L\in t_\id{n}(I)$, then $J\in t_\id{mn}(I)$. Moreover, since $(\id{m}:\id{n})=1$, for every $J\in t_\id{mn}(I)$ there exists a unique $L\in t_\id{n}(I)$ such that $J\in t_\id{m}(L)$, namely the ideal given by $L_\id{p}=I_\id{p}$ for $\id{p}\nmid\id{n}$ and $L_\id{p}=J_\id{p}$ for $\id{p}\mid\id{n}$. Hence
\[
 T_\id{mn}([I])=\sum_{L\in t_\id{n}(I)}\sum_{J\in t_\id{m}(L)}[J]=T_\id{m}(T_\id{n}([I])),
\]
which proves that $T_\id{mn}=T_\id{m} T_\id{n}$.
 \item Let $J\in\mathfrak{I}(R)$. Given $I\in t_{\id{p}^{k+2}}(J)$, write $I_\id{p}=J_\id{p} x_\id{p}$, with $x_\id{p}\in R_r(J_\id{p})$. Then we have a bijection
 \begin{align*}
  R_r(J_\id{p})^\times&\backslash\{y_\id{p}\in R_r(J_\id{p}):v_\id{p}(N(y_\id{p}))=1, x_\id{p} y_\id{p}^{-1} \in R_r(J_\id{p})\} \\ \longrightarrow & \{K\in t_\id{p}(J):I\in t_{\id{p}^{k+1}}(K)\},
 \end{align*}
assigning to each $y_\id{p}$ the ideal $K$ given locally by $K_\id{q}=J_\id{q}$ for $\id{q}\neq\id{p}$ and $K_\id{p}=J_\id{p} y_\id{p}$. Since $\id{p}\nmid \D$ we can identify $R_r(J_\id{p})$ with $M_2(\OO_\id{p})$. By the previous lemma, these sets have one element if $x_\id{p}\notin \pi_\id{p} M_2(\OO_\id{p})$, and $q+1$ elements otherwise. Hence, we have a non-disjoint union
\[
 t_{\id{p}^{k+2}}(J) = \bigcup_{K\in t_\id{p}(J)} t_{\id{p}^{k+1}}(K).
\]
If $I\in t_{\id{p}^{k+2}}(J)$ is such that $x_\id{p}=\pi_\id{p} z_\id{p}$ with $z_\id{p}\in M_2(\OO_\id{p})$, then letting $I'=\id{p}^{-1} I$ we have that $I'\in t_{\id{p}^k}(J)$. Conversely, for each $I'\in t_{\id{p}^k}(J)$ we have that $I=\id{p }I' \in t_{\id{p}^{k+2}}(J)$. Using this, the equality follows easily.

 \item This follows from $(1)$ and $(2)$.
\end{enumerate}

\end{proof}

The Hecke operators are normal with respect to $\ll{\,,\,}$, but not necessarily self-adjoint if $\cl(F)$ is non trivial, as we see in Proposition~\ref{prop:hecke_adj} below.

\medskip

There is an action of the group of fractional ideals on $\mathfrak{I}(R)$. Given a fractional ideal $\id{a}$ and $I\in \mathfrak{I}(R)$, we define $\id{a} I\in \mathfrak{I}(R)$ as the ideal locally given by $(\id{a} I)_\pp=R_\pp (x_\pp \xi_\pp)$, if $\id{a}$ and $I$ are locally given by $\id{a}_\pp=\OO_\pp \xi_\pp$, and $I_\pp=R_\pp x_\pp$, respectively. This induces an action of $\cl(F)$ on $M(R)$, which commutes with the action of the Hecke operators, and which preserves $\ll{\,,\,}$.

\begin{lemma}
Let $I,J\in\mathfrak{I}(R)$. Then, $I\in t_\id{m}(J)$ if and only if $\id{m}J\in t_\id{m}(I)$.
\end{lemma}

\begin{proof}
 
Both statements are equivalent, so we will prove the ``only if'' statement. Let $I\in t_\id{m}(J)$. We prove that $\id{m}J\in t_\id{m}(I)$ by showing that this assertion holds in every completion.

Let $\id{p}$ be a prime ideal. Take $x_\id{p}$ in $R_r(I_\id{p})$ such that $I_\id{p}=J_\id{p}x_\id{p}$. Then, $\id{m}_\id{p}=\OO_\id{p} N(x_\id{p})$. Since $\overline{x_\id{p}}\in R_r(I_\id{p})$, we have that $\id{m}_\id{p}J_\id{p}\subseteq J_\id{p} x_\id{p} \overline{x_\id{p}}\subseteq I_\id{p}$. Furthermore, $[I_\id{p}:\id{m}_\id{p}J_\id{p}]=[J_\id{p}:J_\id{p}\overline{x_\id{p}}]=\id{m}_\id{p}^2$. 

\end{proof}

\begin{proposition}\label{prop:hecke_adj}

The adjoint of $T_\id{m}$ with respect to $\ll{\,,\,}$ is $[\id{m}^{-1}]T_\id{m}$.

\end{proposition}

\begin{proof}
 Let $I,J\in\mathfrak{I}(R)$. Then
\begin{align*}
  \ll{[I],T_\id{m}([J])} & =\sum_{L\in t_\id{m}(J)}\#\{x\in \OO^\times\backslash
B^\times: I x = L\} \\ 
& = \#\{x\in \OO^\times\backslash B^\times: I x \in t_\id{m}(J) \} \\
& = \#\{x\in \OO^\times\backslash B^\times: \id{m} J \in t_\id{m}(Ix) \} \\
& =\#\{x\in \OO^\times\backslash B^\times: \id{m} J x^{-1}\in t_\id{m}(I) \} = \ll{T_\id{m}([I]),[\id{m}][J]},
\end{align*}
where the third equality follows by the previous lemma. This proves the assertion.
\end{proof}

\begin{proposition}
 
 The spaces $E(R)$ and $S(R)$ are preserved by the action of the Hecke operators and by the action of $\cl(F)$.
 
\end{proposition}

\begin{proof}

Let $\id{a},\id{b}$ be fractional ideals. Then
\[
 [\id{b}]e_{[\id{a}]}=\sum_{[I]\in\cl_{[\id{a}]}(R)}\frac{1}{\ll{[\id{b}I],[\id{b}I]}} [\id{b}I]=e_{[\id{b}^2\id{a}]},
\]
since multiplication by $[\id{b}]$ gives a bijection between $\cl_{[\id{a}]}(R)$ and $\cl_{[\id{b}^2\id{a}]}(R)$. This shows that $\cl(F)$ preserves $E(R)$, and hence $S(R)$, since the adjoint of $[\id{a}]$ with respecto to $\ll{\,,\,}$ is $[\id{a}^{-1}]$.

To prove that the Hecke operators preserve $S(R)$, by Proposition~\ref{prop:hecke-cuat} it suffices to prove that $T_\id{p}(S(R))\subseteq S(R)$ for every prime ideal $\id{p}$.

Let $\pp$ be a prime ideal. Given $I\in\mathfrak{I}(R)$, the set $t_\id{p}(I)$ is in bijection with the set
\[
  R_\pp^\times \backslash \left\{x_\pp\in R_\pp: \OO_\pp N(x_\pp)= \pp\OO_\pp\right\},
\]
and hence $\#t_\id{p}(I)=c$ does not depend on $I$. We also note that for $J\in t_\id{p}(I)$ and a fractional ideal $\id{a}$ we have that
\[
 \ll{[J],e_{[\id{a}]}}=\begin{cases}
                            1, & [N(I)]=[\id{p}^{-1}\id{a}], \\
                            0, & [N(I)]\neq[\id{p}^{-1}\id{a}],
                           \end{cases}
\]
since $N(J)=\id{p}N(I)$.

Let $v=\sum_{[I]\in\cl(R)} \lambda_{[I]} [I]\in M(R)$, and let $\id{a}$ be a fractional ideal. Then
\[
 \ll{T_\pp(v),e_{[\id{a}]}} = \sum_{[I]\in\cl(R)} \lambda_{[I]} \Bigg(\sum_{J\in t_\pp(I)} \ll{[J],e_{[\id{a}]}}\Bigg)=c\cdot\Bigg(\sum_{[I]\in\cl_{[\id{p}^{-1}\id{a}]}(R)} \lambda_{[I]}\Bigg)=c\cdot\ll{v,e_{[\id{p}^{-1}\id{a}]}},
\]
which proves that $T_\pp(v)$ is cuspidal if (and only if) $v$ is cuspidal.

Finally, these facts together with Proposition~\ref{prop:hecke_adj} imply that $E(R)$ is preserved by the Hecke operators.

\end{proof}

Since the Hecke operators are commuting, normal operators, $S(R)$ has a basis of simultaneous eigenvectors for the whole Hecke algebra. However, since the operators $T_{\pp^k}$ with $\id{p}\mid\D$ do not satisfy the same relations as the Hecke operators on Hilbert modular forms, we will be interested only in the algebra of operators $\mathbb{T}_0$ generated by the $T_\id{p}$ with $\pp\nmid\D$.

\section{Hilbert modular forms of half-integral weight}\label{sec:3}

We follow \cite{shim-hhalfint} closely, though omitting and avoiding many technical details which are not relevant for our purposes.

As in the rational case, half-integral weight Hilbert modular forms are defined in terms of the theta function
\[
 \theta(z)=\sum_{\xi\in\OO} e_F(\xi^2, z/2), \quad z\in\HH^\mathbf{a}.
\]
We let $J(\gamma,z)=\left(\dfrac{\theta(\gamma z)}{\theta(z)}\right)j(\gamma,z)$ for $\gamma\in \SL_2(F)$.

Let $\id{b}\subseteq\OO$ be an ideal divisible by $4$. Let $\psi$ be a Hecke character of $F$ with conductor dividing $\id{b}$, and denote by $\psi^*$ the character on ideals prime to $\id{b}$ induced by $\psi$. For an integral ideal $\id{m}$ we denote $\psi_\id{m}=\prod_{\pp\mid \id{m}}\psi_\pp$. We also denote $\psi_\mathbf{a} = \prod_{\tau\in\mathbf{a}}\psi_\tau$.

Given fractional ideals $\id{r},\id{n}$, let $\Gamma[\id{r},\id{n}]  = \SL_2(F)\cap\tilde\Gamma[\id{r},\id{n}]$. For $\gamma=\big(\begin{smallmatrix} a & b \\ c & d \end{smallmatrix}\big)\in \SL_2(F)$ and $f:\HH^\mathbf{a}\to\C$, let 
\[
(f\vert\gamma)(z)=\psi_\id{b}(a)^{-1} J(\gamma,z)^{-1}f(\gamma z). 
\]
A \emph{Hilbert modular form} of weight $\mathbf{3/2}$, level $\id{b}$ and character $\psi$, is an holomorphic function $f$ on $\HH^\mathbf{a}$ satisfying
\[
 f\vert\gamma= f \quad \forall\,\gamma\in \Gamma[2^{-1}\id{d},\id{b}].
\]
The space of such $f$ is denoted by $M_\mathbf{3/2}(\id{b},\psi)$. It is trivial unless $\psi_\mathbf{a}(-1)=(-1)^d$.

\medskip

Given $f\in M_\mathbf{3/2}(\id{b},\psi)$, there is a Fourier series attached to each ideal class in $F$. More precisely, for every $\xi\in F$ and every fractional ideal $\id{m}$ there is a complex number $\lambda(\xi,\id{m},f)$ such that
\[
 f(z)=\sum_{\xi\in F}\lambda(\xi,\OO,f)e_F(\xi,z/2),
\]
and such that
\begin{align*}
  \lambda(\xi b^2,\id{m},f) &= N_{F/\Q}(b)\psi_\mathbf{a}(b)\lambda(\xi ,b\id{m},f)
\quad \forall\, b\in F^\times,\\
 \lambda(\xi,\id{m},f) &=0, \quad \text{unless }\xi\in(\id{m}^{-2})^+\cup\{0\}.
\end{align*}
See \cite[Proposition 3.1]{shim-hhalfint}. The Fourier coefficients $\lambda(\xi,\id{m},f)$ for non-princi\-pal $\id{m}$ are harder to describe, but this can be done explicitly in the case of forms given by theta series, which we will consider below.

We say that $f$ is a \emph{cusp form} if $\lambda(0,\id{m},f\vert\gamma)=0$ for every fractional ideal $\id{m}$, for every $\gamma\in \SL_2(F)$. The space of such $f$ is denoted by $S_\mathbf{3/2}(\id{b},\psi)$.

\begin{definition}
 
The \emph{Kohnen plus space} $M^+_\mathbf{3/2}(\id{b},\psi)$ is the subspace of those $f\in M_\mathbf{3/2}(\id{b},\psi)$ such that $\lambda(\xi,\OO,f)=0$ for every $\xi\in\OO^+$ such that $-\xi$ is not a square modulo $4\OO$. We denote $S^+_\mathbf{3/2}(\id{b},\psi)=M^+_\mathbf{3/2}(\id{b},\psi)\cap S_\mathbf{3/2}(\id{b},\psi)$.
 
\end{definition}

In \cite{shim-hhalfint} there are defined Hecke operators for square-free ideals $\id{m}$. Due to normalization issues, in this article we denote by $T_\id{m}$ the $\id{m}$-th Hecke operator of \cite{shim-hhalfint} multiplied by $N(\id{m})$. We recall the action of the Hecke operators in terms of Fourier coefficients (see \cite[Proposition 5.4]{shim-hhalfint}).

\begin{proposition}\label{prop:hecke-coef-me}
 
Let $f\in M_\mathbf{3/2}(\id{b},\psi)$, and let $\pp\nmid\id{b}$. Let $\id{m}$ be a fractional ideal, and take $c_\id{p}\in F_\id{p}$ such that 
$c_\id{p}\OO_\id{p}=\id{m}_\id{p}$. Then,
\[
 \lambda(\xi,\id{m},T_{\pp}(f)) = N(\id{p})\lambda(\xi,\pp\id{m},f) + \psi^*(\pp)\big(\tfrac{\xi c_\id{p}^2}{\pp} \big)\lambda(\xi,\id{m},f) + \psi^*(\pp^2)\lambda(\xi,\pp^{-1}\id{m},f),
\]
where $\big(\frac{*}{\pp}\big)$ denotes the quadratic residue symbol modulo
$\pp$.
\end{proposition}

For $\id{n}\subseteq\OO$ we introduce a formal symbol $M(\id{n})$, satisfying that $M(\id{n}\id{m})=M(\id{n})M(\id{m})$ for all $\id{n},\id{m}\subseteq\OO$. Then we can consider the ring of formal series in these symbols, indexed by integral ideals. The following result, which is essentially \cite[Theorems 6.1 and 6.2]{shim-hhalfint}, is the generalization of the Shimura correspondence for Hilbert modular forms. We assume for simplicity that $\psi$ is a quadratic character, since this will be the case in our setting.

\begin{theorem}\label{thm:shimura-map}

For each $\xi\in\OO^+$ there is a Hecke linear map $\Shim_\xi:M_\mathbf{3/2}(\id{b},\psi) \to M_\mathbf{2}(\id{b}/2)$, characterized by the following property. Write $\xi\OO=\qq^2\id{r}$ with $\qq,\id{r}$ integral and $\id{r}$ square-free, and let $\epsilon_\xi$ be the Hecke character corresponding to $F(\sqrt{\xi})/F$. Let $f\in M_\mathbf{3/2}(\id{b},\psi)$. Then (formally), 
\begin{align*}
& \sum_{\id{m}\subseteq\OO} c(\id{m} ,\Shim_\xi(f))M(\id{m}) \\
& =\left(\sum_{\id{m}\subseteq\OO}\lambda(\xi,\qq^{-1}\id{m},f)M(\id{m}
)\right)\left(\sum_{\id{m}\subseteq\OO}(\psi\epsilon_\xi)^*(\id{m})N(\id{m})^{
-1}M(\id{m}) \right).
\end{align*}

\end{theorem}

\section{Ternary theta series}\label{sec:4}

Theta series of totally definite ternary quadratic forms can be used to construct Hilbert modular forms of weight $\mathbf{3/2}$, as we show in Proposition~\ref{prop:theta-tern} below. We start this section by recalling some results from \cite{shim-hhalfint} that we need to prove this Proposition.

We denote by $F_\mathbb{A}$ the ring of adeles of $F$, and we denote by $F_\mathbf{f}$ the non-archimedean part of $F_\A$. Given a fractional ideal $\id{n}$, we denote
\[
 \Gamma_\A[\id{n}]=\Big\{\big(\begin{smallmatrix} a & b \\ c & d       
\end{smallmatrix}\big)\in \SL_2(F_\A):a_\pp,d_\pp\in \OO_\pp, b_\pp\in (2\id{d}^{-1})_\pp, c_\pp\in (2\id{n}\id{d})_\pp \quad \forall\,\pp \Big\},
\]
Given a $2\times2$ matrix $\beta$, we use the notation $\beta=\big(\begin{smallmatrix} a_\beta & b_\beta \\ c_\beta & d_\beta       
\end{smallmatrix}\big)$ to refer to the coefficients of $\beta$. For $\beta\in \SL_2(F)$, we denote by $\id{a}_\beta$ the fractional ideal given locally by $(\id{a}_\beta)_\pp=(c_\beta)\id{d}^{-1}_\pp + d_\beta \OO_\pp$.

\medskip

We fix a  totally negative definite matrix a $S\in M_3(F)$. Denote by $\psi$ the Hecke character corresponding to the quadratic extension $F(\sqrt{\det S})/F$, and denote by $\id{f}$ its conductor.

Let $\mathcal{S}(F_\mathbf{f}^3)$ denote the Schwartz-Bruhat space of locally constant functions on $F_\mathbf{f}^3$. Given $\eta\in \mathcal{S}(F_\mathbf{f}^3)$, we consider the theta series given by
\[
 g(z;\eta)=\sum_{\xi\in F^3} \eta(\xi) e_F(\xi S\xi^t,\tfrac{\overline{z}}{2}).
\]

In \cite[Proposition 2.4]{shim-hhalfint} there is defined an action of $\SL_2(F)$ on $\mathcal{S}(F_\mathbf{f}^3)$, which is denoted by $(\beta,\eta)\mapsto {^{\beta}\eta}$. In terms of this action we have the following transformation formula for $g(z;\eta)$.

\begin{proposition}\label{prop:transf-theta}
 Denote by $P(F_\A)$ the subgroup of $\SL_2(F_\A)$ of upper triangular matrices. For every $\beta\in \SL_2(F)\cap P(F_\A)\Gamma_\A[\OO]$,
\[
 g(\beta z;{^{\beta}\eta})=\overline{J(\beta,z)} g(z;\eta).
\]
\end{proposition}

\begin{proof}
 This is \cite[Proposition 11.4]{shim-hhalfint}. Note that since $S$ is totally negative definite, the automorphy factor $J_S$ involved in that result is given by
 \[
  J_S(\beta,z)=h(\beta,z)\cdot \vert j(\beta,z) \vert ^3 j(\beta,z)^{-3}.
 \]
 It satisfies that $\overline{J_S}=J$, since by \cite[(2.19b)]{shim-hhalfint} we have that $j^2=h^4$.
\end{proof}

The following two results, which are respectively Propositions 11.7 and 11.5 from \cite{shim-hhalfint}, show how $\SL_2(F)$ acts on $\mathcal{S}(F_\mathbf{f}^3)$ in certain cases.

\begin{proposition}\label{prop:eta}
 Given $\eta\in\mathcal{S}(F_\mathbf{f}^3)$, let $M$ be an $\OO$-lattice in $F^3$ such that $\eta(x+u)=\eta(x)$ for every $u\in M$. Furthermore, let $\id{r},\id{n},\id{z}$ be fractional ideals of $F$ satisfying:
 \begin{enumerate}
  \item $x S x^t\in\id{r}\,$ for every $x\in F^3$ such that $\eta(x)\neq0$.
  \item $x S x^t\in\id{n}\,$ for every $x\in F^3$ such that $\Tr(xSy^t)\in\id{d}^{-1}$ for every $y\in M$.
  \item $\eta(xa)=\eta(x)$ for every $a\in \hat\OO^\times$ such that $a_\pp-1\in \id{z}_\pp\,$ for every $\pp$.
 \end{enumerate}
 Let $\id{a}=\id{r}^{-1}\cap\OO$ and $\id{b}=4\D\cap\id{z}\cap 4\id{a}\cap 4\id{d}^{-1}\id{a}\id{n}^{-1}$. Then
 \[
  {^{\beta}\eta}(x)=\psi_\id{f}(d_\beta) \eta(x(a_\beta)_\id{z}) \qquad\forall\, \beta\in \Gamma[2^{-1}\id{d}\id{a}^{-1},\id{b}],
 \]
 where $(a_\beta)_\id{z}$ denotes the projection of $a_\beta$ to $\prod_{\pp\mid\id{z}}F_\pp^\times$.
\end{proposition}

\begin{proposition}\label{prop:transf-eta}
 Given $\eta\in\mathcal{S}(F_\mathbf{f}^3)$, there is an open subgroup $U$ of $\Gamma_\A[\id{f}]$ such that if $\beta\in \SL_2(F)\cap\left(\begin{smallmatrix} t & 0\\ 0 & t^{-1}\end{smallmatrix}\right) U$ with $t\in F_\mathbf{f}^\times$, then
 \[
  {^{\beta}\eta}(x)=\psi_\mathbf{a} (d_\beta)\psi^*(d_\beta \id{a}_\beta^{-1})N(\id{a}_\beta)^{3/2} \eta(xt) \qquad \forall\,x\in F_\mathbf{f}^3.
 \]
\end{proposition}

We now apply these results to our setting. Let $B$ be a totally definite quaternion algebra over $F$. For $x\in B$ denote $\Delta(x)=\Tr(x)^2-4N(x)$, the \emph{discriminant} of $x$. Let $V=B/F$, and for $x\in B$ denote by $[x]$ its class in $V$. Then $\Delta$ determines an integral, totally negative definite quadratic form on $V$. For $I\in \mathfrak{I}(R)$, we consider $R_r(I)/\OO$ as a lattice in $V$, which we denote by $L_I$.

From here on, let $\psi$ be the Hecke character corresponding to the extension $F(\sqrt{-1})/F$. This quadratic character has conductor $\id{f}$ dividing $4\OO$, and the corresponding ideal character satisfies $\psi^*(\pp)=(\tfrac{-1}{\pp})$ for $\pp\nmid2$. By local class field theory, $\psi$ satisfies the equality $\psi_\mathbf{a}(-1)=(-1)^d$. Hence, the space $M_\mathbf{3/2}(4\D,\psi)$ is not trivially zero.

\begin{proposition}\label{prop:theta-tern}
 
Given $I\in \mathfrak{I}(R)$, let
\[
 \vartheta_I(z)=\sum_{x\in L_I} e_F\big(-\Delta(x),\tfrac{z}{2}\big).
\]
Then $\vartheta_I\in M^+_\mathbf{3/2}(4\D,\psi)$. Furthermore, the Fourier coefficients of $\vartheta_I$ are given by
\[
\lambda(\xi,\id{a},\vartheta_I)=N(\id{a})^{-1}\cdot\#\{[x]\in\id{a}^{-1}L_I:-\Delta(x)=\xi\big\}.
\]

\end{proposition}

\begin{proof}

Let $\{v_1,v_2,v_3\}$ be a basis of $V$, and let $\id{a}_1,\id{a}_2,\id{a}_3$ be fractional ideals such that $L_I=\oplus_{i=1}^3 \id{a}_i v_i$. Through this basis we identify $V$ with $F^3$. Let $S$ be the matrix of the quadratic form $\Delta$ with respect to this basis. If $B\simeq\left<1,i,j,k : i^2=a,\, j^2= b,\, ij=-ji=k\right>_F$, then the determinant of $\Delta$ with respect to the basis $\{[i],[j],[k]\}$ equals $-(8 a b)^2$. This shows that $\det(S)=-1\in F^\times / (F^\times)^2$.

Let $\eta\in \mathcal{S}(F_\mathbf{f}^3)$ be the characteristic function of $M=\id{a}_1\oplus\id{a}_2\oplus\id{a}_3$. Then the theta series $g(z;\eta)$ defined by $S$ and $\eta$ satisfies that
\begin{equation}\label{eqn:g=theta}
 g(z;\eta)=\sum_{\xi\in F^3} \eta(\xi) e_F(\Delta(\xi),\tfrac{\overline{z}}{2})=\overline{\vartheta_I(z)}.
\end{equation}

The function $\eta$ satisfies the hypotheses of Proposition~\ref{prop:eta}, taking $\id{r}=\id{z}=\OO$, and $\id{n}=\id{d}^{-2} \D^{-1}$. The first two assertions are clear. To prove the last equality, take $[x]\in V$ such that $[x] S [y]^t \in \id{d}^{-1}$ for every $[y]\in L_I$. Assume, without loss of generality, that $\Tr(x)=0$. Then, a simple calculation shows that $2\Tr(x\overline{y})\in \id{d}^{-1}$ for every $y\in R_r(I)$. Hence, by \cite[Lemma 1.2.5]{gebh}, we have that $\Delta([x])=-N(2x)\in \id{d}^{-2} \D^{-1}$.

Then Propositions~\ref{prop:transf-theta} and~\ref{prop:eta} together with \eqref{eqn:g=theta} give that
\[
 \vartheta_I(\beta z) = \psi_\id{f}^{-1}(d_\beta) J(\beta,z) \vartheta_I(z) \quad \forall\,\beta\in  \Gamma[2^{-1}\id{d},4\D].
\]
Since $\psi$ is quadratic and its conductor $\id{f}$ divides $4\D$, we have that $\psi_\id{f}^{-1}(d_\beta)=\psi_{4\D}(a_\beta)$ for all $\beta \in\Gamma[2^{-1}\id{d},4\D] $. This proves that $\vartheta_I\in M_\mathbf{3/2}(4\D,\psi)$. To see that it belong to the Kohnen plus space, note that 
\[
 \lambda(\xi,\OO,\vartheta_I)=\#\{[x]\in L_I:-\Delta(x)=\xi\big\}
\]
equals $0$ if $-\xi$ is not a square modulo $4\OO$.

We now consider the Fourier coefficients of $\vartheta_I$. Given a fractional ideal $\id{a}$, take $t\in F_\mathbf{f}^\times$ such that $t\OO=\id{a}$. Let $\beta\in \SL_2(F)$ be as in Proposition~\ref{prop:transf-eta}. Since $\beta = \left(\begin{smallmatrix} t & 0\\ 0 & t^{-1}\end{smallmatrix}\right) q$ with $q\in \Gamma_\A[\id{f}]$, we have that $\id{a}_\beta=t^{-1}\OO=\id{a}^{-1}$. Then by \cite[(3.14c)]{shim-hhalfint} we have that
\begin{equation}\label{eqn:coefff1}
 \psi_\mathbf{a}(d_\beta)\psi^*(d_\beta \id{a})
J(\beta,\beta^{-1}z) \vartheta_I(\beta^{-1}z) = N(\id{a})^{-1/2} \sum_{\xi\in F} \lambda(\xi,\id{a},\vartheta_I) e_F(\xi,z/2).
\end{equation}

On the other hand, by Propositions~\ref{prop:transf-theta} and~\ref{prop:transf-eta}, we have that
\begin{align}\label{eqn:coefff2}
 J(\beta,\beta^{-1}z)\vartheta_I(\beta^{-1}z) & = \overline{g(z;^{\beta}\eta)} \nonumber \\
& = \psi_\mathbf{a} (d_\beta)\psi^*(d_\beta\id{a}) N(\id{a})^{-3/2}\sum_{\xi\in F^3} \eta(\xi t)e_F(\Delta(\xi),\tfrac{z}{2}). 
\end{align}
Since the map $\xi\mapsto\eta(\xi t)$ equals $1$ if $\xi\in \id{a}^{-1} L_I$ and $0$ otherwise, comparing \eqref{eqn:coefff1} and \eqref{eqn:coefff2} yields the desired equality.

\end{proof}

We now prove that this construction is $\TT_0$-linear. For this, we start with the following auxiliary result.

\begin{lemma}\label{lemma:cuentas-coeff}
 Let $\id{p}$ be a prime ideal such that $\id{p}\nmid 4\D$. Let $[x]\in \id{p}^{-1} L_I$. Then,
 \[
  \#\{J\in t_\id{p}(I): [x]\in L_J\}=\begin{cases}
                                      1+N(\id{p}), \quad & [x]\in \id{p} L_I, \\
                                      1+\big(\tfrac{\Delta(x)}{\pp}
\big), \quad & [x]\in L_I\setminus \id{p} L_I, \\
0 \text{ or }1, \quad & [x]\in \id{p}^{-1} L_I \setminus L_I.
                                     \end{cases}
 \]
 
\end{lemma}

\begin{proof}
 
 Note that given $J\in t_\id{p}(I)$, we have that $[x]\in L_J$ if and only if $[x]\in (L_J)_\id{p}$, since $(L_I)_\id{q}=(L_J)_\id{q}$ for every $\id{q}\neq\id{p}$. Since $\id{p}\nmid 4\D$ we can identify $R_r(I)_\id{p}$ with $M_2(\OO_\id{p})$. Then, the set $\{J\in t_\id{p}(I): [x]\in L_J\}$ is in bijection with the set
\[
\mathfrak{X}=\SL_2(\OO_\id{p}) \backslash \{
y_\id{p}\in M_2(\OO_\id{p}): \det y_\id{p}=\pi_\id{p},\, y_\id{p} x_\id{p}
y_\id{p}^{-1}\in F_\id{p}+M_2(\OO_\id{p})\},
\]
letting to each such $y_\id{p}$ correspond the ideal $J\in \mathfrak{I}(R)$ given locally by $J_\id{q} = I_\id{q}$ for $\id{q} \neq \id{p}$, and $J_\id{p}=I_\id{p} y_\id{p}$.

To compute the set $\mathfrak{X}$, we use the same system of representatives for the action of $\SL_2(\OO_\id{p})$ in $\{ y_\id{p}\in M_2(\OO_\id{p}): \det y_\id{p}=\pi_\id{p}\}$ as in Lemma~\ref{lemma:nro-orbitas}. We start by considering the first two cases. Assume then that $x\in R_r(I)$. Write $x_\id{p}=\left(\begin{smallmatrix} a & b\\ c & d \end{smallmatrix}\right)\in M_2(\OO_\id{p})$. Then, we have that
\begin{align*}
\left(\begin{matrix} \pi_\id{p} & 0 \\ 0 & 1\end{matrix}\right) x_\id{p}
\left(\begin{matrix} \pi_\id{p} & 0 \\ 0 & 1\end{matrix}\right)^{-1}
\in F_\id{p}+M_2(\OO_\id{p})
& \Longleftrightarrow \pi_\id{p}\mid c, \\
\left(\begin{matrix} 1 & \alpha \\ 0 & \pi_\id{p}
\end{matrix}\right) x_\id{p} \left(\begin{matrix} 1 & \alpha \\ 0 & \pi_\id{p}
\end{matrix}\right)^{-1} \in F_\id{p}+ M_2(\OO_\id{p})
& \Longleftrightarrow  \pi_\id{p}\mid -c \alpha^2 + (d-a)\alpha + b.\\
\end{align*}
If $x_\id{p}\in \OO_\id{p}+\id{p} M_2(\OO_\id{p})$, we see that $\mathfrak{X}$ has $1+N(\id{p})$ elements. If $x_\id{p}\notin \OO_\id{p}+\id{p} M_2(\OO_\id{p})$, let $P=-cX^2 + (d-a)X + b\in k_\id{p}[X]$. Then $P\neq0$, and its discriminant equals $(d-a)^2+4bc=\Delta(x_\id{p})$. Hence $\mathfrak{X}$ has $1+\Big(\frac{\Delta(x_\id{p})}{\pp}\Big)$ elements.

Now consider the case when $[x]\in \id{p}^{-1} L_I \setminus L_I$. Assume then that $\pi_\id{p} x_\id{p}=\left(\begin{smallmatrix} a & b \\ c & d \end{smallmatrix}\right)\in M_2(\OO_\id{p})$, and that $x_\id{p}\notin F_\id{p}+M_2(\OO_\id{p})$. Then, we have that
\begin{align}
\label{eqn:cuentafea1}\left(\begin{matrix} \pi_\id{p} & 0 \\ 0 &
1\end{matrix}\right) x_\id{p}
\left(\begin{matrix} \pi_\id{p} & 0 \\ 0 & 1\end{matrix}\right)^{-1}
\in F_\id{p}+M_2(\OO_\id{p})
& \Longleftrightarrow \begin{cases}
\pi_\id{p}^2\mid c,\\\pi_\id{p}\mid d-a,                       
                      \end{cases}
 \\
\label{eqn:cuentafea2}\left(\begin{matrix} 1 & \alpha \\ 0 & \pi_\id{p}
\end{matrix}\right) x_\id{p} \left(\begin{matrix} 1 & \alpha \\ 0 & \pi_\id{p}
\end{matrix}\right)^{-1} \in F_\id{p}+ M_2(\OO_\id{p})
& \Longleftrightarrow\begin{cases}
\pi_\id{p}^2\mid -c \alpha^2 + (d-a)\alpha +
b,\\\pi_\id{p}\mid (d-a)-2c\alpha.                      
                     \end{cases}
\end{align}

 Suppose that \eqref{eqn:cuentafea1} holds, and that there exists $\alpha$ such that \eqref{eqn:cuentafea2} holds. Then $\pi_\id{p}\mid c,d-a,b$, thus contradicting the fact that $x_\id{p}\notin F_\id{p}+M_2(\OO_\id{p})$. Finally, assume that there exist distinct $\alpha_1,\alpha_2$ such that \eqref{eqn:cuentafea2} holds. Then, substracting equations we see that $\pi_\id{p}\mid 2c$. If $\pi_\id{p}\mid c$ we have that $\pi_\id{p}\mid d-a,b$, which again is not possible. If $\pi_\id{p}\mid 2$, we have that $\pi_\id{p}\mid d-a$, and hence the polynomial $P$ defined above has null discriminant. This is a contradiction, since $P$ has $\alpha_1,\alpha_2$ as roots. Thus, we have proved that $\mathfrak{X}$ has at most one element, which completes the proof.
 \end{proof}

For $\xi\in F^+\cup\{0\}$, a fractional ideal $\id{a} $ and $I\in\mathfrak{I}(R)$, denote
\[
 a(\xi,\id{a},[I])=\#\{[x]\in \id{a}^{-1}L_I:-\Delta(x)=\xi\big\}.
\]
Let $e_\xi\in M(R)$ be given by
\[
 e_\xi=\sum_{[J]\in \cl(R)} \tfrac{a(\xi,\OO,[J])}{\ll{[J],[J]}}  [J].
\]
Note that $e_0=\sum_{[\id{a}]\in\cl(F)} e_{[\id{a}]}$.

\begin{theorem}\label{teo:theta-H}
 
Given $v\in M(R)$, let
\[
 \theta(v)(z)=\sum_{\xi\in\OO^+\cup\{0\}} \ll{e_\xi,v}
e_F\big(\xi,\tfrac{z}{2}\big), \quad z\in \HH^\mathbf{a}.
\]
Then, $\theta(v)\in M^+_\mathbf{3/2}(4\D,\psi)$, and $\theta(v)$ is cuspidal if and only if $v$ is cuspidal. Furthermore, the map $\theta$ is $\TT_0$-linear.

\end{theorem}

\begin{proof}
First assume that $v=[I]$, with $I\in \mathfrak{I}(R)$. Then $\theta([I])=\vartheta_I$, which implies the first claim. To prove the Hecke linearity, let $\id{p}$ be a prime ideal not dividing $4\D$. Let $f=\theta(T_\id{p}([I]))$. Since
\[
 f=\sum_{\xi\in\OO^+\cup\{0\}} \Big(\sum_{J\in t_\id{p}(I)} \ll{e_\xi,[J]}\Big) e_F\big(\xi,\tfrac{z}{2}\big).
\]
we have that
\[
 \lambda(\xi,\OO,f)=\#\{(J,[x])\in \mathfrak{I}(R)\times V: J\in
t_\id{p}(I),\, [x]\in L_J,\, -\Delta(x)=\xi\}.
\]
To compute the size of this set, we use Lemma~\ref{lemma:cuentas-coeff}, considering the following three possibilities for those $[x]\in V$ for which there exists $J\in t_\id{p}(I)$ such that $[x]\in L_J,\, -\Delta(x)=\xi$. Note that since $\id{p} I\subseteq J \subseteq I$ for $J\in t_\id{p}(I)$, then every such $[x]$ belongs to $\id{p}^{-1} L_I$.
\begin{itemize}
\item $[x]\in \id{p} L_I$. There are $a(\xi,\id{p}^{-1},[I])$ such $[x]$, and for each one there are $1+N(\id{p})$ ideals $J$ as above.
 \item $[x]\in L_I\setminus\id{p} L_I$. There are $a(\xi,\OO,[I])-a(\xi,\id{p}^{-1},[I])$ such $[x]$, and for each one there are $1+\big(\tfrac{\Delta(x)}{\pp}\big)$ ideals $J$ as above. Note that $a(\xi,\id{p}^{-1},[I])\big(\frac{\Delta(x)}{\pp}\big)=0$, since if there exists $[y]\in \id{p}L_I$ such that $\Delta([y])=\xi$, then $\id{p}\mid \xi$.
 \item $[x]\in\id{p}^{-1} L_I\setminus L_I$. There are $a(\xi,\id{p},[I])-a(\xi,\OO,[I])$ such $[x]$, and for each one there is just one ideal $J$ as above.
\end{itemize}

Adding up, using Propositions \ref{prop:hecke-coef-me} and \ref{prop:theta-tern} we see that
\begin{align*}
\lambda(\xi,\OO,f)&=a(\xi,\id{p}^{-1},[I])\big(1+N(\id{p})\big)+ \\
&+
\big(a(\xi,\OO,[I])-a(\xi,\id{p}^{-1},[I])\big)\Big(1+\big(\tfrac{\Delta(x)}{\pp
} \big)\Big) \\
& + a(\xi,\id{p},[I])-a(\xi,\OO,[I])\\
&=N(\id{p})a(\xi,\id{p}^{-1},[I])+a(\xi,\OO,[I])\big(\tfrac{\Delta(x)}{\pp}
\big)+a(\xi,\id{p},[I]) \\
&=\lambda(\xi,\id{p}^{-1},\vartheta_I)+\big(\tfrac{\xi}{\pp}\big)\psi^*(\id{p}
)\lambda(\xi,\OO,\vartheta_I)+N(\id{p})\lambda(\xi,\id{p},\vartheta_I)\\
&=\lambda(\xi,\OO,T_\id{p}(\vartheta_I)),
\end{align*}
which proves that $T_\id{p}(\theta([I]))=\theta(T_\id{p}([I]))$.

Finally, suppose that $v=\sum_{[I]}\lambda_{[I]} [I]$ is cuspidal. Then $\theta(v)$ is cuspidal at infinity, since $\ll{e_0,v}=0$. Since $\theta(v)=\sum_{[I]}\lambda_{[I]} \vartheta_I$ is a linear combination of theta series corresponding to quadratic forms in the same genus with $\sum_{[I]}\lambda_{[I]}=0$, it is cuspidal. This is a classical result by Siegel, generalized to the totally real field setting in \cite{walling}.

\end{proof}

\section{Computing preimages}\label{sec:5}

We now consider the problem of computing preimages of the Shimura map. This is, given $\xi\in\OO^+$, and given a newform $g$ of weight $\mathbf{2}$, to construct a form $f$ of weight $\mathbf{3/2}$ such that $\Shim_\xi(f)=g$.

\medskip

Let $\id{c}$ be an integral ideal, and suppose that $B$ is a totally definite quaternion algebra having an Eichler order $R$ of discriminant $\id{c}$ (such $B$ exists if $d$ is even, or if $d$ is odd and $\id{c}$ is not a square).

\begin{proposition}\label{prop:cuspidalidad}
 
Let $v\in S(R)$. Then, $\Shim_\xi(\theta(v))$ is a cusp form.
 
\end{proposition}

\begin{proof}
We can assume that $v$ is a $\TT_0$-eigenvector. Denote $g=\Shim_\xi(\theta(v))$. Then if for $\pp\nmid\id{c}$ we let $\omega_\pp$ denote the $\pp$-th eigenvalue of $v$, since the maps $\theta$ and $\Shim_\xi$ are $\TT_0$-linear, we have that $T_\pp g=\omega_\pp g$.

By the theory of Hilbert Eisenstein series, for which we refer to \cite{wiles} and \cite{atwill}, it suffices to prove that $g$ is orthogonal to every Eisenstein series $E$.

Let $\pp\nmid\id{c}$. We have that $T_\pp E=c(\pp,E)E$ (see \cite[Proposition 3,3]{atwill}). Then, the self-adjointness of the Petersson inner product implies that
\[
 \omega_\pp\ll{g,E}=c(\pp,E)\ll{g,E}.
\]
This implies that $\ll{g,E}=0$, since by \cite{shahidi} we have that $\vert \omega_\pp \vert\leq 2N(\pp)^{7/10}$, whereas by the definition of $E$ (see \cite[Proposition 3.1]{atwill}) we have that $\vert c(\pp,E)\vert\geq N(\pp)-1$.

We finish by remarking that though in \cite{atwill} the authors consider weights $\mathbf{k}\geq\mathbf{3}$, the results we used are still valid in weight $\mathbf{2}$ when $F\neq \Q$. The case $F=\Q$ follows by the same arguments, taking special care with the definition of the Eisenstein series of weight $2$ (see \cite{wiles}).
\end{proof}

We have then the following diagram of $\TT_0$-linear maps, where $J-L$ denotes the Jacquet-Langlands correspondence:
\[
\xymatrix{
S(R)\ar[rr]^{J-L} \ar[dr]^{\theta} & & S_{\mathbf{2}}(\id{c})
\ar@{..>}@/_/[dl]
\\ & S^+_{\mathbf{3/2}}(4\id{c},\psi) \ar@/_/[ur]_{\Shim_\xi}}.
\]

\begin{remark}\label{rmk:nivel/4}
 
According to Theorem~\ref{thm:shimura-map}, the map $\Shim_\xi$ in principle divides the level by $2$, and hence for $v\in S(R)$, the form $\Shim_\xi(\theta(v))$ has level $\id{2c}$ instead of the level $\id{c}$ claimed in the diagram. In the classical setting, when $\id{c}$ is odd and square-free, (a small part of) the theory of Kohnen asserts that when applied to forms in the Kohnen plus space, the Shimura map divides the level by $4$. In the setting of Hilbert modular forms, the theory of the Kohnen plus space is currently under development by Hiraga and Ikeda. The case when $\id{c}=\OO$ has been achieved in \cite{ikeda}, and the general (odd, square-free) case is expected to be developed soon.
 
\end{remark}

We summarize this discussion in the next theorem, which is the main result of this article.

\begin{theorem}\label{thm:teo_ppal}
 Let $g\in S_{\mathbf{2}} (\id{c})$ be a newform, and let $v_g \in S(R)$ be a $\TT_0$-eigenvector with the same eigenvalues as $g$. Let $\tilde{g}=\Shim_\xi(\theta(v_g))$. If $\tilde{g}$ has level $\id{c}$, then $\tilde{g}$ is a multiple of $g$.
\end{theorem}

\begin{proof}
 First, note that such $v_g$ exists (and is unique) due to the Jacquet-Langlands theory (see for example \cite[Proposition 2.12]{Hida}). Since the operators $\theta$ and $\Shim_\xi$ are $\TT_0$-linear, the cusp form $\tilde{g}$ has the same eigenvalues as $g$, and then by the multiplicity one result given in \cite{miyake} we have that $\tilde{g}$ is a multiple of $g$.
\end{proof}

\begin{remark}
 It could happen that $\tilde{g}$ is the zero cusp form. Nevertheless, for odd and square-free $\id{c}$, the theory of the Kohnen space asserts that:
 \begin{itemize}
  \item A linear combination of the maps $\Shim_\xi$ is an isomorphism between the new subspace of $S^+_{\mathbf{3/2}}(4\id{c},\psi)$ and the new subspace of $S_{\mathbf{2}} (\id{c})$ (which in particular implies that there exists $\xi$ such that $\Shim_\xi(\theta(v_g))\neq0$).
  \item If $\theta(v_g)$ is not zero, then $\theta(v_g)$ is a newform mapping to a non-zero multiple of $g$ under this isomorphism, by a strong multiplicity one result in $S^+_{\mathbf{3/2}}(4\id{c},\psi)$.
 \end{itemize}
\end{remark}

\begin{remark}\label{rmk:shim-iso}
 We expect $\tilde{g}$ to have level $\id{c}$. Since we know that in the worst case it has level $\id{2c}$, then it must be a linear combination of $g(z)$ and $g(2z)$. In any given example, this combination can be found in terms of Fourier coefficients, and we can verify that $\tilde{g}$ has actually level $\id{c}$ by seeing that the coefficient corresponding to $g(2z)$ is null.
\end{remark}

The main issue is then to know whether there exists a quaternion algebra $B$ and an Eichler order $R$ such that $\theta(v_g) \neq 0$. We assume from now on that $\id{c}$ is odd and square-free.

\medskip

The following conjecture is just a naive generalization to the Hilbert setting of the result due to B\"ocherer and Schulze-Pillot for classical modular forms of odd and square-free level (see \cite[page 378]{BSP}).

\begin{conjecture}\label{conj:notzero}
The form $\theta(v_g)$ is non zero if and only if $L(g,1) \neq 0$ and the quaternion algebra $B$ ramifies exactly at the archimedean primes and at all primes $\id{p}$ dividing $\id{c}$ where the Atkin-Lehner involution $W_\pp$ acts on $g$ with eigenvalue $w_\pp=-1$.
\end{conjecture}

If $L(g,1)\neq0$, the functional equation safisied by $L(g,s)$ implies that $(-1)^d\prod_{\pp\mid\id{c}} w_\pp=1$. Then an algebra $B$ as in the conjecture exists, and it is unique up to isomorphism.

\medskip

The relation between Fourier coefficients and central values of twisted $L$-series is given by the following theorem, which was proved for classical forms in \cite[page 378]{BSP} and in a more general setting for Hilbert modular forms in \cite{mao} and \cite{shim-coef}, generalizing Waldspurger's results over $\Q$. 

\begin{theorem}\label{thm:schulze-pill}
Let $g \in S_{\mathbf{2}}(\id{c})$ be a newform such that $f= \theta(v_g)$ is non-zero. Let $\xi\in\OO^+$ be such that $-\xi$ is a fundamental discriminant. Let $\epsilon_\xi$ be the Hecke character corresponding to $F(\sqrt{-\xi})/F$. Then
\begin{equation}\label{eqn:waldspurger}
 \vert\lambda(\xi,\OO,f)\vert^2 = \kappa  L(g,\epsilon_{\xi},1)\prod_{\id{p} \mid \id{c}}(c(\id{p},g)-\epsilon_\xi(\id{p})),
\end{equation}
where $\kappa$ is a non-zero constant, and $L(g,\epsilon_{\xi},s)$ is the twist of the L-series of $g$ by $\epsilon_{\xi}$.
\end{theorem}

In particular, under the above assumptions, this conjecture states that $L(g, \epsilon_{\xi}, 1)= 0$ if and only if $\lambda(\xi,\OO,f)=0$, if $\xi$ is such that the product over $\pp\mid\id{c}$ in the right hand side of \eqref{eqn:waldspurger} is non-zero.
\section{An example}\label{secc:6}

We let $F=\Q(\sqrt{5})$, which has trivial narrow class group, and we denote $\omega=\tfrac{1+\sqrt{5}}{2}$. We let $E$ be the elliptic curve over $F$ given by
\[
 E: \quad y^2+xy+\omega y = x^3 -(1+\omega) x^2.
\]
This curve has prime conductor, equal to $\id{c}=(5+2\omega)$, and satisfies that $L(E,1)\neq0$. The space $M_{\mathbf{2}}(\id{c})$ has dimension 2, and it is generated by an Eisenstein series and a newform $g$ which corresponds to $E$. Its first eigenvalues are given in \cite{lass-expl}; we only state that $c(\id{c},g)=-1$. According to Conjecture~\ref{conj:notzero}, we choose $B$ to be the unramified totally definite algebra over $F$, i.e. the algebra generated by $1,i,j,k$, where $i^2=j^2=-1,ij=k=-ji$. If $R$ is an Eichler order of discriminant $\id{c}$ in $B$, then the theory of Jacquet-Langlands asserts that there exists $v\in S(R)$ which is an eigenvector for $\TT_0$ with the same eigenvalues as $g$.

Using the algorithm presented in \cite{ariel-yo}, with the aid of \texttt{SAGE} (\cite{sage}), we obtain the desired order, which is given by
\[
 R= \left<\dfrac{1-(\omega+1)j-(\omega+10) k}{2}, \dfrac{ i- \omega j + -(\omega+21) k}{2}, j-5k,(5\omega-3)k \right>_\OO.
\]
This order has class number equal to 2, and hence there is no need to compute the Hecke operators, since $S(R)$ is $1$-dimensional. A set of representatives for the set of $R$-ideal classes is given by $R$ and the ideal $I$ given by
\[
 I= \left<\dfrac{1-(\omega+1)j-(\omega+38) k}{2}, \dfrac{ i- \omega j + -(\omega+49) k}{2}, j+3k,(5\omega-3)k \right>_\OO.
\]
We have that $v=[R]-[I]$ is an eigenvector for the whole Hecke algebra, since $\deg(v)=0$.

Let $f=\theta(v)$. We consider $L_R$ and $L_I$ as lattices of dimension $6$ over $\Z$, and use a lattice basis reduction algorithm on the integral, positive definite quadratic form $\Tr_{F/\Q}\circ(-\Delta)$ to compute the Fourier coefficients $\lambda(\xi,\OO,f)$, with $\Tr_{F/\Q}(\xi)\leq 100$ and $-\xi$ a fundamental discriminant. We find that there are non-zero coefficients, thus verifying Conjecture~\ref{conj:notzero}. The zero coefficients split into two families, which we consider below.

\begin{itemize}
\item {\bf The trivial zeros} are the ones such that
\[
 \lambda(\xi,\OO,\theta([R]))=\lambda(\xi,\OO,\theta([I]))=0.
\]
For this zeros Theorem \ref{thm:schulze-pill} is easy to verify. The local-global principle for quadratic forms implies that the non existence of points $x\in L_R \cup L_I$ with $-\Delta(x)=\xi$ is equivallent to the equality $\epsilon_\xi(\id{c})= -1$, so in this case both sides of \eqref{eqn:waldspurger} vanish trivially.
\item {\bf The non-trivial zeros} are the ones such that 
\[
 \lambda(\xi,\OO,\theta([R]))=\lambda(\xi,\OO,\theta([I]))\neq0.
\]
For these zeros, we have that $\epsilon_\xi(\id{c})=1$, and hence by \eqref{eqn:waldspurger} we have that $L(g,\epsilon_{\xi},1)=0$. The non-trivial zeros with $\Tr_{F/\Q}(\xi) \le 100$ are
\[
35+8w, 39+15w, 47-9w, 51-5w, 62-27w.
\]
For these $\xi$, the Birch and Swinnerton-Dyer conjecture predicts that the rank of the quadratic twist of $E$ by $-\xi$ should be positive (and even, because the sign of the functional equation equals $1$). We verified using $2$-descent that all these curves have rank equal to $2$.
\end{itemize}

\section*{Acknowledgements}
I would like to thank Gonzalo Tornar\'ia, and specially Ariel Pacetti, for their invaluable support.


\begin{thebibliography}{Dem05}

\bibitem[AL13]{atwill}
Timothy~W. Atwill and Benjamin Linowitz.
\newblock Newform theory for {H}ilbert {E}isenstein series.
\newblock {\em Ramanujan J.}, 30(2):257--278, 2013.

\bibitem[BM07]{mao}
Ehud Moshe Baruch and Zhengyu Mao.
\newblock Central value of automorphic $L$-functions.
\newblock{\em Geom. Funct. Anal.} 17(2):333-384, 2007.

\bibitem[BSP90]{BSP}
Siegfried B{\"o}cherer and Rainer Schulze-Pillot.
\newblock On a theorem of {W}aldspurger and on {E}isenstein series of {K}lingen
  type.
\newblock {\em Math. Ann.}, 288(3):361--388, 1990.

\bibitem[Dem05]{lass-expl}
Lassina Demb{\'e}l{\'e}.
\newblock Explicit computations of {H}ilbert modular forms on {${\mathbb
  Q}(\sqrt{5})$}.
\newblock {\em Experiment. Math.}, 14(4):457--466, 2005.

\bibitem[DV10]{dem-voi}
Lassina Demb{\'e}l{\'e} and John Voight.
\newblock Explicit methods for {H}ilbert modular forms. \newblock 2010.
\newblock{\url {http://arxiv.org/abs/1010.5727}}.

\bibitem[Gar90]{garrett}
Paul~B. Garrett.
\newblock {\em Holomorphic {H}ilbert modular forms}.
\newblock The Wadsworth \& Brooks/Cole Mathematics Series. Wadsworth \&
  Brooks/Cole Advanced Books \& Software, Pacific Grove, CA, 1990.

\bibitem[Geb09]{gebh}
Ute Gebhardt.
\newblock Explicit construction of spaces of {H}ilbert modular cusp forms using
  quaternionic theta series.
\newblock 2009.
\newblock Thesis (Ph.D.)--Universitat des Saarlandes.

\bibitem[Gro87]{gross}
Benedict~H. Gross.
\newblock Heights and the special values of {$L$}-series.
\newblock In {\em Number theory ({M}ontreal, {Q}ue., 1985)}, volume~7 of {\em
  CMS Conf. Proc.}, pages 115--187. Amer. Math. Soc., Providence, RI, 1987.

\bibitem[HI13]{ikeda}
Kaoru Hiraga and Tamotsu Ikeda.
\newblock On the {K}ohnen plus space for {H}ilbert modular forms of
  half-integral weight {I}.
\newblock {\em Compos. Math.}, 149(12):1963--2010, 2013.
  
\bibitem[Hid81]{Hida}
Haruzo Hida.
\newblock On abelian varieties with complex multiplication as factors of the
  {J}acobians of {S}himura curves.
\newblock {\em Amer. J. Math.}, 103(4):727--776, 1981.

\bibitem[Koh82]{kohnen-newforms}
Winfried Kohnen.
\newblock Newforms of half-integral weight.
\newblock {\em J. Reine Angew. Math.}, 333:32--72, 1982.

\bibitem[Miy71]{miyake}
Toshitsune Miyake.
\newblock On automorphic forms on {${\rm GL}_{2}$} and {H}ecke operators.
\newblock {\em Ann. of Math. (2)}, 94:174--189, 1971.

\bibitem[PS12]{ariel-yo}
Ariel Pacetti and Nicol\'as Sirolli.
\newblock Computing ideal classes representatives in quaternion algebras.
\newblock 2012.
\newblock \url{http://arxiv.org/abs/1007.2821}.
\newblock To appear in \textit{Mathematics of Computation}.

\bibitem[PT07]{tornaria}
Ariel Pacetti and Gonzalo Tornar{\'{\i}}a.
\newblock Shimura correspondence for level {$p^2$} and the central values of
  {$L$}-series.
\newblock {\em J. Number Theory}, 124(2):396--414, 2007.

\bibitem[S11]{sage}
W.\thinspace{}A. Stein et~al.
\newblock {\em {S}age {M}athematics {S}oftware ({V}ersion 4.7)}.
\newblock The Sage Development Team, 2011.
\newblock \url{http://www.sagemath.org}.

\bibitem[Sha90]{shahidi}
F.~Shahidi.
\newblock Best estimates for {F}ourier coefficients of {M}aass forms.
\newblock In {\em Automorphic forms and analytic number theory ({M}ontreal,
  {PQ}, 1989)}, pages 135--141. Univ. Montr\'eal, Montreal, QC, 1990.

\bibitem[Shi73]{shim-halfint}
Goro Shimura.
\newblock On modular forms of half integral weight.
\newblock {\em Ann. of Math. (2)}, 97:440--481, 1973.

\bibitem[Shi75]{shintani}
Takuro Shintani.
\newblock On construction of holomorphic cusp forms of half integral weight.
\newblock {\em Nagoya Math. J.}, 58:83--126, 1975.
  
\bibitem[Shi78]{shim-special}
Goro Shimura.
\newblock The special values of the zeta functions associated with {H}ilbert
  modular forms.
\newblock {\em Duke Math. J.}, 45(3):637--679, 1978.

\bibitem[Shi87]{shim-hhalfint}
Goro Shimura.
\newblock On {H}ilbert modular forms of half-integral weight.
\newblock {\em Duke Math. J.}, 55(4):765--838, 1987.

\bibitem[Shi93]{shim-coef}
Goro Shimura.
\newblock On the {F}ourier coefficients of {H}ilbert modular forms of
  half-integral weight.
\newblock {\em Duke Math. J.}, 71(2):501--557, 1993.

\bibitem[Vig80]{vig}
Marie-France Vign{\'e}ras.
\newblock {\em Arithm\'etique des alg\`ebres de quaternions}, volume 800 of
  {\em Lecture Notes in Mathematics}.
\newblock Springer, Berlin, 1980.

\bibitem[Wal91]{waldspurger}
Jean-Loup Waldspurger.
\newblock Correspondances de Shimura et quaternions.
\newblock {\em Forum Math.}, 3:219--307, 1991.

\bibitem[Wal94]{walling}
Lynne~H. Walling.
\newblock A remark on differences of theta series.
\newblock {\em J. Number Theory}, 48(2):243--251, 1994.

\bibitem[Wil86]{wiles}
A.~Wiles.
\newblock On {$p$}-adic representations for totally real fields.
\newblock {\em Ann. of Math. (2)}, 123(3):407--456, 1986.

\bibitem[Xue11]{xue}
Hui Xue.
\newblock Central values of {$L$}-functions and half-integral weight forms.
\newblock {\em Proc. Amer. Math. Soc.}, 139(1):21--30, 2011.

\end{thebibliography}
\end{document}